\documentclass{conm-p-l}

\usepackage{amssymb}
\usepackage{mathrsfs}
\usepackage{tikz}
\usepackage{soul}
\usepackage{xcolor}
\usepackage{physics}

\newtheorem{theorem}{Theorem}[section]
\newtheorem{lemma}[theorem]{Lemma}

\theoremstyle{definition}
\newtheorem{definition}[theorem]{Definition}
\newtheorem{example}[theorem]{Example}

\newtheorem*{problem}{Problem}

\theoremstyle{remark}
\newtheorem{remark}[theorem]{Remark}

\numberwithin{equation}{section}

\begin{document}

\title{The Racah algebra: An overview and recent results}


%    author one information
\author[H. De Bie]{Hendrik De Bie}
\address{Department of Electronics and Information Systems, Faculty of Engineering and Architecture, Ghent University, Krijgslaan 281, 9000 Ghent, Belgium}
\email{hendrik.debie@ugent.be}
\thanks{}

%    author two information
\author[P. Iliev]{Plamen Iliev}
\address{School of Mathematics, Georgia Institute of Technology, Atlanta, GA, 30332-0160, USA}
\email{iliev@math.gatech.edu}
\thanks{}

\author[W. van de Vijver]{Wouter van de Vijver}
\address{Department of Electronics and Information Systems, Faculty of Engineering and Architecture, Ghent University, Krijgslaan 281, 9000 Ghent, Belgium}
\email{Wouter.vandeVijver@UGent.be}

\author[L. Vinet]{Luc Vinet}
\address{Centre de Recherches Math\'ematiques, Universit\'e de Montr\'eal, P.O. Box 6128, Centre-ville Station, Montr\'eal, QC H3C 3J7, Canada}
\email{vinet@crm.umontreal.ca}

\subjclass[2000]{Primary }
%    The 2010 edition of the Mathematics Subject Classification is
%    now available.  If you are citing a classification from the
%    new scheme, use the following input coding instead.
%\subjclass[2010]{Primary }

\date{\today}

\begin{abstract}
Recent results on the  Racah algebra $\mathcal{R}_n$ of rank $n - 2$ are reviewed. $\mathcal{R}_n$  is defined in terms of generators and relations and sits in the centralizer of the diagonal action of $\mathfrak{su}(1,1)$ in $\mathcal{U}(\mathfrak{su}(1,1))^{\otimes n}$. Its connections with multivariate Racah polynomials are discussed. It is shown to be the symmetry algebra of the generic superintegrable model on the $ (n-1)$ - sphere and a number of interesting realizations are provided.
\end{abstract}

\maketitle

\section{Introduction}

It is understood since the seminal work of Zhedanov \cite{Zhedanov-1991} that the bispectral properties of the families of orthogonal polynomials of the Askey scheme can be encoded into algebras that are generally quadratic. Basically, the generators of these algebras, which bear the names of the different families, are realized by differential or difference operators. These operators have the eponym polynomials as eigenfunctions by acting on the variable or the degree of these polynomials. The operators acting on the variable can be differential or difference depending on the family of polynomials. The difference operators acting on the degree coincide with recurrence operators. When multiplied these operators are taken to be realized in the same representation; that is either acting on the variable or the degree. This is how the Racah algebra (of rank 1) was originally defined.

%Basically, the generators of these algebras which bear the names of the different families are realized by the differential or difference operator of which the polynomials are eigenfunctions and the recurrence operator. (These operators are obviously taken in the same representation that is either the variable or the degree one.) This is how the Racah algebra (of rank 1) was originally defined.

The Racah polynomials are also known to enter in the overlaps between bases associated to the recoupling of three $\mathfrak{su}(1,1)$ (or $\mathfrak{su}(2)$) irreducible representations; these coefficients are referred to as $6j$-symbols. Since these bases are defined by diagonalizing intermediate Casimir elements in $\mathcal{U}(\mathfrak{su}(1,1))^{\otimes 3}$, this naturally led to the observation that $\mathcal{R}_3$ also arises in  the centralizer of the diagonal action of $\mathfrak{su}(1,1)$ in $\mathcal{U}(\mathfrak{su}(1,1))^{\otimes 3}$ with the generators here represented by these intermediate Casimir operators \cite{GZ}. This result naturally paves the way for the construction of higher rank generalizations of $\mathcal{R}_3$ by considering instead of $3$ an arbitrary number of 
$\mathfrak{su}(1,1)$ factors. These are the algebras on which we will focus in this review.

The Racah algebra $\mathcal{R}_3$ was further seen to occur in a number of interesting situations. For instance, by taking the $\mathfrak{su}(1,1)$ copies as realizations of
 the dynamical algebra of the singular oscillator, one observes that the total Casimir operator is essentially the Hamiltonian of the generic superintegrable system on the $2$-sphere \cite{Kalnins&Miller&Post-2007}. The picture of $\mathcal{R}_3$ as a commutant that we described in the preceding paragraph readily leads to the conclusion \cite{Genest&Vinet&Zhedanov-2014-2} that the Racah algebra is the symmetry algebra of this superintegrable Hamiltonian. $\mathcal{R}_3$ was also found to possess an embedding into $\mathcal{U}(\mathfrak{su}(1,1))$ \cite{Gao&Wang&Hou-2013}. When using models for $\mathfrak{su}(1,1)$ in terms of differential operators, this last embedding gives a one-variable realization of $\mathcal{R}_3$ while the embedding in
$\mathcal{U}(\mathfrak{su}(1,1))^{\otimes 3}$ provides a three-variable one. The connection between these two models was obtained  in \cite{Genest&Vinet&Zhedanov-2014} through separation of variables.

Roughly five years ago, one of us co-authored a paper \cite{Genest&Vinet&Zhedanov-2014-3} to introduce the Racah algebra and review many of the facets we mentioned. Since then the subject has much evolved especially regarding the higher rank generalizations. The present paper offers a timely overview of these advances. Interestingly the review \cite{Genest&Vinet&Zhedanov-2014-3} contained a picture very much like the one below that we feel appropriate to reproduce here. 

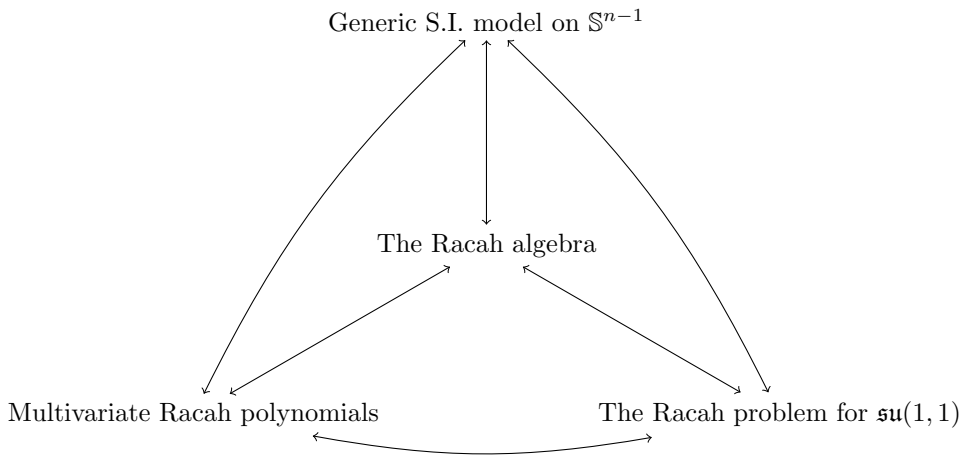
\begin{center}
\begin{tikzpicture}
	\node (Algebra) at (0,0) {The Racah algebra};
	\node (OP) at (210:4.5cm) {Multivariate Racah polynomials};
	\node (Problem) at (330:4.5cm) {The Racah problem for $\mathfrak{su}(1,1)$};
	\node (model) at (90:3cm) {Generic S.I. model on $\mathbb{S}^{n-1}$ };
	\draw[<->] (OP) to [bend left=10] (model);
	\draw[<->] (model) to [bend left=10] (Problem);
	\draw[<->] (Problem) to [bend left=10] (OP);
	\draw[<->] (Algebra) to (OP);
	\draw[<->] (Algebra) to (Problem);
	\draw[<->] (Algebra) to (model);
\end{tikzpicture}
\end{center}

It depicts the Racah algebra as a central entity connected to orthogonal polynomials, recoupling problems and superintegrable systems. It  further indicates that the interrelations between these topics
can thus be understood on the basis of the common underlying algebraic structure.

This diagram also offers a nice snapshot  of the outline of the present article which will unfold as follows. The Racah algebra $\mathcal{R}_n$ of rank $(n-2)$ is defined in Section \ref{Rn} in the framework of a generalized Racah problem, as the subalgebra of $\mathcal{U}(\mathfrak{su}(1,1))^{\otimes n}$ generated by the intermediate Casimir elements in this tensor product. All commutation relations will be given. Section \ref{S3} describes the connections between 
$\mathcal{R}_n$ and multivariate Racah polynomials. These will appear as overlaps between two representation bases diagonalized by different labelling Abelian subalgebras. The case when the polynomials defined by Tratnik occur \cite{Tratnik-1991} will be pointed out. With an eye to applications, Section \ref{S4} will present four different realizations of $\mathcal{R}_n$. First, upon realizing each $\mathfrak{su}(1,1)$ as the conformal algebra in one dimension, it will be explained how $\mathcal{R}_n$ arises as the symmetry algebra of the generic superintegrable model on the $(n-1)$ - sphere. 
Second a realization of $\mathcal{R}_n$ in terms of Dunkl operators \cite{DTAMS} will be provided. Third, contrasting the realizations of $\mathcal{R}_n$ in terms of $n$ variables in the superintegrable model context or in terms of Dunkl operators, it will be shown how differential operators in $(n-2)$ variables that satisfy the commutation relations of $\mathcal{R}_n$ can be constructed by calling upon the Barut-Girardello realization of $\mathfrak{su}(1,1)$. Fourth, the loop will be closed if one bears in mind how the Racah algebra was initially identified. Indeed Section \ref{S4} will end by extracting directly from properties of the multivariate polynomials a realization of  $\mathcal{R}_n$ in terms of the difference operators of which the multivariate Racah polynomials are eigenfunctions \cite{Geronimo&Iliev-2010}. Section \ref{Conclusion} will offer concluding remarks.

\section{The higher rank Racah algebra}\label{Rn}
Let $\mathfrak{su}(1,1)$ be the Lie algebra generated by the operators $J_{+}$, $J_{-}$ and $J_0$ obeying the following relations
\[
 [J_0,J_\pm]=\pm J_\pm, \quad [J_-,J_+]=2J_0.
\]
Its Casimir is given by 
\[
	C:=J_0^2-J_0-J_+J_-.
\]
This operator sits inside the universal enveloping algebra $\mathcal{U}(\mathfrak{su}(1,1))$. We consider now the $n$-fold tensor product $\mathcal{U}(\mathfrak{su}(1,1))^{\otimes n}$. In this algebra we define the following operators with $\epsilon=\pm $ or $0$:
\begin{align*}
J_{\epsilon,k}=\underbrace{1 \otimes \dots \otimes 1}_{k-1 \text{ times}} \otimes J_{\epsilon} \otimes \underbrace{1 \otimes \dots \otimes 1}_{n-k \text{ times}} 
\end{align*}
Let $K$ be a subset of $[n]:=\{1, \dots, n \}$. We define:
\[
 J_{\epsilon,K}=\sum_{k \in K} J_{\epsilon, k}.
\]
The following lemma is easy to check:
\begin{lemma}
Let $K \subset [n]$. The operators $J_{0,K}$, $J_{+,K}$ and $J_{-,K}$ generate an algebra isomorphic to $\mathfrak{su}(1,1)$. We denote this algebra by $\mathfrak{su}^K(1,1)$.
\end{lemma}
The algebra $\mathfrak{su}^K(1,1)$ lives in the components of $\mathcal{U}(\mathfrak{su}(1,1))^{\otimes n}$ whose index is in $K$. Consider its Casimir:
\[
 C_{K}:=J_{0,K}^2-J_{0,K}-J_{+,K}J_{-,K}.
\]
The Casimirs of all possible $\mathfrak{su}^K(1,1)$ generate the higher rank Racah algebra.
\begin{definition}
The higher rank Racah algebra $\mathcal{R}_n$ of rank $n-2$ is the subalgebra of $\mathcal{U}(\mathfrak{su}(1,1))^{\otimes n}$ generated by the following set of operators:
\[
  \{ C_K \,|\, K \subset [n] \text{ and } K\neq \emptyset \}.
\]
\end{definition}
We exclude the empty set as $C_{\emptyset}=0$. 
\begin{remark}
Alternatively, one can construct the generators $C_A$ from the comultiplication $\mu^*$ of $\mathfrak{su}(1,1)$. This is an algebra morphism that embeds $\mathfrak{su}(1,1)$ into the tensor product $\mathfrak{su}(1,1) \otimes \mathfrak{su}(1,1)$. It is defined as follows on the generators:
\begin{equation}
\label{comultiplication} 
\mu^*(J_0)=J_0\otimes 1+1\otimes J_0, \qquad \mu^*(J_\pm)=J_\pm\otimes 1+1\otimes J_\pm .
\end{equation}
This map extends to the universal enveloping algebra $\mathcal{U}(\mathfrak{su}(1,1))$. This allows us to apply the comultiplication repeatedly on the Casimir $C$.
$$
\mathscr{C}_1:=C, \qquad  \mathscr{C}_n:=(\underbrace{1\otimes\ldots\otimes 1}_{n-2 \text{ times }} \otimes \mu^*)(\mathscr{C}_{n-1}).
$$
Each Casimir $\mathscr{C}_k$ lives in $\mathcal{U}(\mathfrak{su}(1,1))^{\otimes k}$. We want to lift this Casimir to  $\mathcal{U}(\mathfrak{su}(1,1))^{\otimes n}$. Consider the map
$$
\tau_k: \bigotimes_{i=1}^{m-1} \mathcal{U}(\mathfrak{su}(1,1))\rightarrow \bigotimes_{i=1}^{m} \mathcal{U}(\mathfrak{su}(1,1)),
$$
which acts as follows on homogeneous tensor products:
$$
\tau_k(t_1\otimes \ldots \otimes t_{m-1}):=t_1\otimes \ldots \otimes t_{k-1} \otimes 1 \otimes t_k \otimes \ldots \otimes t_{m-1} .
$$
and extend it by linearity.  The map $\tau_k$ adds a $1$ at the $k$-th place. This allows to define the following:
\begin{align}
\label{Casimir-Upper}
C_A:=\left(\prod_{k \in \left[ n\right] \backslash A}^{\longrightarrow} \tau_k \right)\left(\mathscr{C}_{|A|}\right).
\end{align}
with $A$ a subset of $[n]$. This leads to the same generators for the Racah algebra $\mathcal{R}_n$. This alternative pathway has been successfully applied on other (super) Lie algebras and their $q$-deformations. See for example \cite{DeBie&DeClercq&vandeVijver,DeBie&Genest$vandeVijver&Vinet-2}. For a more general approach see \cite{Lehrer&Zhang}.
\end{remark}

\begin{lemma}\label{sl2}(See also \cite{Lehrer&Zhang})
The higher rank Racah algebra  sits in the centralizer of $\mathfrak{su}^{[n]}(1,1)$. 
\end{lemma}
\begin{proof}We have
\[
[C_A,J_{\epsilon,[n]}]=[C_A,J_{\epsilon,A}]+[C_A,J_{\epsilon,[n]\backslash A}]=0.
\]
The first commutator equals $0$ as $C_A$ is the Casimir of $\mathfrak{su}^A(1,1)$ and the second commutator is $0$ as both operators act on different components of $\mathcal{U}(\mathfrak{su}(1,1))^{\otimes n}$.
\end{proof}

\begin{example}\label{R3}
We call the simplest non-trivial example $\mathcal{R}_3$  the rank one case. It is generated by the set
\[
 \{ C_1, C_2, C_3, C_{12}, C_{23}, C_{13},C_{123} \}.
\]
For ease of notation we abbreviate sets of the form $\{1,2\}$ by $12$ when in the index of a generator. The elements $C_1$, $C_2$, $C_3$ and $C_{123}$ are central elements of $\mathcal{R}_n$. A tedious calculation shows that the generators are not linearly independent
\begin{equation}\label{LinDep}
C_{123}=C_{12}+C_{13}+C_{23}-C_1-C_2-C_3 .
\end{equation}
As $C_{123}$ is central and by formula \eqref{LinDep} we have
\begin{align*}
0=[C_{123},C_{12}]=[C_{13},C_{12}]+[C_{23},C_{12}], \\
0=[C_{123},C_{13}]=[C_{12},C_{13}]+[C_{23},C_{13}], \\
0=[C_{123},C_{23}]=[C_{12},C_{23}]+[C_{13},C_{23}]. 
\end{align*}
We conclude that $[C_{12},C_{23}]=[C_{23},C_{13}]=[C_{13},C_{12}]$. We introduce the following operator
\[
F:=\frac{1}{2}[C_{12},C_{23}] =\frac{1}{2}[C_{23},C_{13}] =\frac{1}{2}[C_{13},C_{12}] .
\]
Another tedious computation shows the following relations to be true:
\begin{align}
\begin{split}\label{Rel1}
 [C_{12},F]=C_{23}C_{12}-C_{12}C_{13}+(C_2-C_1)(C_3-C_{123}), \\
 [C_{23},F]=C_{13}C_{23}-C_{23}C_{12}+(C_3-C_2)(C_1-C_{123}), \\
 [C_{13},F]=C_{12}C_{13}-C_{13}C_{23}+(C_1-C_3)(C_2-C_{123}).
 \end{split}
\end{align}
\end{example}

\subsection{Relations for $\mathcal{R}_n$}
We wish to find relations for the higher rank Racah algebra $\mathcal{R}_n$ for general $n$. To do so we mention the following lemma:
\begin{lemma}\label{embed}
Let $\{ K_p \}_{p=1..k}$ be a set of $k$ disjoint subsets of $[n]$. Define $K_B:=\cup_{q \in B} K_q$ with $B \subset [k]$. The following map is an injective morphism:
\[
	\theta: \mathcal{R}_k \rightarrow \mathcal{R}_n\,:\, C_B \rightarrow C_{K_B}.
\]
The image of $\theta$ is denoted by $\mathcal{R}_k^{K_1,\dots,K_k}$.
\end{lemma}
\begin{proof}
See \cite[section 4.2]{DeBie&Genest$vandeVijver&Vinet} or follow the strategy given in \cite[Lemma 2.4]{Crampe&vandeVijver&Vinet} .
\end{proof}

\begin{example}
The sets $K_1=\{2\}$, $K_2=\{1,4\}$ and $K_3=\{3\}$ are disjoint subsets of the set $\{1,2,3,4\}$. By Lemma \ref{embed} we have an injective morphism of $\mathcal{R}_3$ into $\mathcal{R}_4$. Explicitly it is given as follows:
\begin{align*}
 &\theta(C_{1})= C_{K_1}=C_{2},& \theta(C_{12})&=C_{K_1K_2}=C_{124},\\
 &\theta(C_{2})= C_{K_2}=C_{14},& \theta(C_{13})&=C_{K_1K_3}=C_{23},\\
 &\theta(C_{3})= C_{K_3}=C_3,& \theta(C_{23})&=C_{K_2K_3}=C_{134}, \\
 &\theta(C_{123})=C_{K_1K_2K_3}=C_{1234}. & &
\end{align*}

\end{example}

By Lemma \ref{embed} we can lift the relations of $\mathcal{R}_3$ given in Example \ref{R3} to $\mathcal{R}_n$. Let $K$, $L$ and $M$ be three disjoint subsets of $[n]$ and consider equality $\eqref{LinDep}$. By Lemma \ref{embed} we replace $1$ by $K$, $2$ by $L$ and $3$ by $M$:
\[
 C_{KLM}=C_{KL}+C_{KM}+C_{LM}-C_{K}-C_{L}-C_{M}.
\]
As before the notation $KL$ is short for $K \cup L$ in the index of a generator. We have found a set of linear dependencies between the generators of $\mathcal{R}_n$.  By induction one can prove the following:
\begin{lemma}\label{LinearDependence} For any set $K \subset [n]$, it holds that
\begin{equation*}
 C_K=\sum_{\left\{i,j\right\}\subset K} C_{ij}-\left(|K|-2\right)\sum_{i \in K} C_i .
\end{equation*}
\end{lemma}

In $\mathcal{R}_3$ we also know that $C_1$ is central. In particular we have $[C_1,C_{12}]=0$ and $[C_1,C_2]=0$. By Lemma \ref{embed} we find $[C_K, C_{KL}]=0$ and $[C_K,C_L]=0$. The following lemma follows:
\begin{lemma}\label{Commutative}
Let $A$ and $B$ be subsets of $[n]$. If either $A \subset B$, $B \subset A$ or $A \cap B=\emptyset$ then $[C_A,C_B]=0$.
\end{lemma}
A consequence of this lemma is that the generators $C_i$, $i \in [n]$ and $C_{[n]}$ are central in $\mathcal{R}_n$.
This lemma shows the existence of many Abelian subalgebras. 
\begin{definition}\label{label}
Consider the following chain $\mathcal{A}$ of subsets of $[n]$:
\[
A_1 \subset A_2 \subset \dots \subset A_{n-2}
\]
with $|A_k|=k+1$. We define the labeling Abelian algebra $\mathcal{Y}_\mathcal{A}$ to be the algebra generated by $\{C_{A_i}\,|\, i=1\dots n-2\}$. 
\end{definition}
We exclude the sets with one element and all the elements because the related generators $C_i$ and $C_{[n]}$ are central. Observe that the number of generators in a labeling Abelian algebra $\mathcal{Y}_\mathcal{A}$ equals the rank of the higher rank Racah algebra.

In what follows we will give the commutator of any pair of generators. It suffices to do this only for the generators with two indices $C_{ij}$ and one index $C_i$ by Lemma \ref{LinearDependence}. The generators with one index are central so we focus on the generators with two indices. We introduce the following operator:
\[
F_{ijk}=\frac{1}{2}[C_{ij},C_{jk}].
\]
The order of the indices of $F$ is important.  Switching any two indices leads to change in sign:
\[
F_{jik}=\frac{1}{2}[C_{ji},C_{ik}]=\frac{1}{2}[C_{kj},C_{ji}]=-\frac{1}{2}[C_{ji},C_{kj}]=-\frac{1}{2}[C_{ij},C_{jk}]=-F_{ijk}.
\]

We apply Lemma $\ref{embed}$ on relation \eqref{Rel1}. We are focusing on generators with two indices so  we set $K=\{i\}$, $L=\{j\}$ and $M=\{k\}$:
\[
[C_{jk},F_{ijk}]=C_{ik}C_{jk}-C_{jk}C_{ij}+(C_k-C_j)(C_i-C_{ijk}).
\]
The remaining possible commutators are found through straightforward but tedious computations. For notational purposes we introduce 
\[
 P_{ij}=C_{ij}-C_i-C_j.
\]
We find:
\begin{align*}
 [P_{kl},F_{ijk}]&=P_{ik}P_{jl}-P_{il}P_{jk}, \\
 [F_{ijk},F_{jkl}]&=F_{jkl}P_{ij}-F_{ikl}(P_{jk}+2C_j)-F_{ijk}P_{jl},\\
 [F_{ijk},F_{klm}]&=F_{ilm}P_{jk}-P_{ik}F_{jlm}.
\end{align*}

\section{The Racah problem for $\mathfrak{su}(1,1)$ and multivariate Racah polynomials}\label{S3}
Racah problems play an important role in algebras related to quantum systems. Famous solutions to the Racah problem for $\mathfrak{sl}_2$ or equivalently $\mathfrak{su}(1,1)$ are the Clebsch-Gordan coefficients, the $3j$-symbols, the $6j$-symbols and in general the $3nj$-symbols for $\mathfrak{su}(1,1)$. We will cast these problems into a single framework. Let us first pose the problem we want to solve.
\begin{problem} The Racah problem for $\mathfrak{su}(1,1)$ is the following:
Let $V$ be an irreducible representation of $\mathcal{R}_n$. Consider two different labeling Abelian algebras $\mathcal{Y}_1$ and $\mathcal{Y}_2$. Let the set $\{ \psi_k \}$ be a basis of $V$ diagonalized by $\mathcal{Y}_1$ and $\{ \varphi_s\} $ be a basis of V diagonalized by $\mathcal{Y}_2$. What are the connection coefficients  between these to bases? In other words, find the numbers $R_{sk}$ such that
\[
\sum_{k} R_{sk} \psi_k=\varphi_s.
\]
\end{problem}
We will consider finite dimensional representations.
The solution to the Racah problem for the rank one Racah algebra $\mathcal{R}_3$ has been known for a long time, see \cite{Gao&Wang&Hou-2013,Genest&Vinet&Zhedanov-2014-2,Genest&Vinet&Zhedanov-2014-3,Granovskii&Zhedanov-1988}. We will give the result here. In the rank one case we have three labeling Abelian algebras: $\mathcal{Y}_1=\langle C_{12} \rangle$,  $\mathcal{Y}_2=\langle C_{23} \rangle$, and  $\mathcal{Y}_3=\langle C_{13} \rangle$. The connection coefficients between $\mathcal{Y}_1$ and $\mathcal{Y}_2$ are Racah polynomials.

\begin{definition} (\cite{Koekoek&Lesky&Swarttouw-2010})
Let $r_n(\alpha,\beta,\gamma, \delta; x)$ be the classical univariate Racah polynomials
 \begin{align*}
  	r_n(\alpha,\beta,\gamma, \delta; x)&:=(\alpha+1)_n(\beta+\delta+1)_n(\gamma+1)_n \,
  	{}_4F_3 \left[ \substack{ -n, n+\alpha+\beta+1, -x, x+\gamma+\delta+1 \\ \alpha+1, \beta+\delta+1, \gamma+1}; 1 \right].
 \end{align*}
\end{definition}
We will give the exact form for the connection coefficients. To this end we first introduce the following polynomial:
\[
\kappa(x,\beta)=\left(x+\frac{\beta+1}{2}\right)\left(x+\frac{\beta-1}{2}\right).
\]
\begin{lemma}\label{overlap}
Let $V$ be an irreducible representation of $\mathcal{R}_n$ with $\dim(V)=N+1$. Assume that the central elements on this representation act as the following scalars:
\begin{align*}
 C_1&=\kappa(0,\beta_0), \\		
 C_2&=\kappa(0,\beta_1-\beta_0-1), \\
 C_3&=\kappa(0,\beta_2-\beta_1-1),\\
 C_{123}&=\kappa(N,\beta_2).
\end{align*}
Assume that $\{\psi_k\}$ is a basis of $V$ diagonalized by $C_{12}$ and $\{ \varphi_s\}$ is a  basis of $V$ diagonalized by $C_{23}$ with the following eigenvalues:
\begin{align*}
C_{12}\psi_k&=\kappa(k,\beta_1)\psi_k, \\
C_{23}\varphi_s&=\kappa(s,\beta_2-\beta_0-1)\varphi_s.
\end{align*}
The connection coefficients are up to a gauge constant equal to 
\[
R_{sk}=r_{s}(\beta_1-\beta_0-1,\beta_{2}-\beta_1-1,-N-1, \beta_1+N;k).
\]
\end{lemma}
The constants $\beta_0$, $\beta_1$ and $\beta_2$ depend on the representation and can be calculated from the action of the central elements. In what follows we will denote the overlap coefficients depending on these central elements:
\[
R_{sk}(C_1,C_2,C_3,C_{123})
\]

Consider now the general case. Let $\mathcal{Y}_{\mathcal{A}_1}$ and $\mathcal{Y}_{\mathcal{A}_2}$ be two labeling Abelian algebras. To find the connection coefficients between bases diagonalized by $\mathcal{Y}_{\mathcal{A}_1}$ and $\mathcal{Y}_{\mathcal{A}_2}$ we will assume that the chains of sets $\mathcal{A}_1$ and $\mathcal{A}_2$ differ by only one element:
\begin{align*}
\mathcal{A}_1: A_1\subset \dots \subset A_{i-1} \subset K \subset A_{i+1} \subset \dots \subset A_{n-2}, \\ 
\mathcal{A}_2: A_1\subset \dots \subset A_{i-1} \subset L \subset A_{i+1} \subset \dots \subset A_{n-2}.
\end{align*}
Let $\{ \psi_{\vec k} \}$ be a basis of the representation $V$ diagonalized by $\mathcal{Y}_{\mathcal{A}_1}$ with $\vec k=(k_1,\dots, k_{n-2})$
\begin{align*}
 C_{A_l}\psi_{\vec k}=\lambda_{k_l}\psi_{\vec k}, \\
 C_{K}\psi_{\vec k}=\lambda_{k_i}\psi_{\vec k}.
\end{align*}
Let $\{ \varphi_{\vec s} \}$ be a basis of the representation $V$ diagonalized by $\mathcal{Y}_{\mathcal{A}_2}$ with $\vec s=(s_1,\dots, s_{n-2})$
\begin{align*}
 C_{A_l}\varphi_{\vec s}=\mu_{s_l}\varphi_{\vec s}, \\
 C_{L}\varphi_{\vec s}=\mu_{s_i}\varphi_{\vec s}.
\end{align*}
We want to describe the connection coefficients $R_{\vec s \vec k}$ such that
\[
\sum_{\vec k} R_{\vec s\vec k}\psi_{\vec k}=\varphi_{\vec s}.
\]
Because both bases $\{ \psi_{\vec k} \}$ and $\{ \varphi_{\vec s} \}$ are diagonalized by the generators $C_{A_l}$ with $l \in [n-2]\backslash \{i\}$ we may assume that $\mu_{s_l}=\lambda_{k_l}$. From this we reduce the connection coefficients as follows
\[
R_{\vec s\vec k}=R_{s_i k_i}\prod_{l\neq i} \delta_{s_l k_l}.
\]
We reduced the problem to finding the coefficients $R_{s_i k_i}$. Consider any eigenspace of the set of operators $\{C_{A_l}\,|\, l \neq i\}$ and denote it by $E$. The operators $C_K$ and $C_L$ commute with the set of operators so they will preserve these eigenspaces. Moreover, these eigenspaces are representations for $\mathcal{R}_3$ by Lemma \ref{embed}, if we consider the embedding morphism by setting.
\[ 1 \rightarrow K\backslash L,\quad 2 \rightarrow K \cap L=A_{i-1}, \quad 3 \rightarrow L \backslash K. \]
This algebra $\mathcal{R}^{K\backslash L,K \cap L, L\backslash K}_3$ is generated by
\[
C_{K\backslash L}, \quad C_{K \cap L},\quad C_{A_{i-1}}, \quad C_{L \backslash K},\quad C_{K}, \quad C_{L}, \quad C_{A_{i+1}}
\]
where $A_{i+1}=K \cup L$ and each of the generators commute with $\{C_{A_l}\,|\, l \neq i\}$ preserving the eigenspaces. Consider the intersection of $\{\psi_{\vec k}\}$ with the eigenspace $E$. This will be a basis of the  eigenspace $E$ diagonalized by $C_{12}$. Equivalently, the intersection of $\{\varphi_{\vec s} \}$ and the same eigenspace $E$ will be a basis of this eigenspace $E$ diagonalized by $C_{23}$. By Lemma \ref{overlap} the connection coefficients between the two bases of the representation $E$ of $\mathcal{R}_3$ will therefore be the Racah polynomials $R_{s_ik_i}(C_{K\backslash L},C_{A_{i-1}}, C_{L\backslash K},C_{A_{i+1}})$. We conclude that the connection coefficients between two bases diagonalized by two labeling Abelian algebra differing by only one generator is given by
\[
 R_{\vec s \vec k}=R_{s_ik_i}(C_{K\backslash L},C_{A_{i-1}}, C_{L\backslash K},C_{A_{i+1}})\prod_{l\neq i} \delta_{s_l k_l}.
\]
Let us now consider two general labeling Abelian algebras. The connection coefficients can be calculated by introducing a sequence of labeling Abelian algebras so that each subsequent pair of algebras only differ by one generator. Each algebra in the sequence will diagonalize a basis of the representation $V$ of $\mathcal{R}_n$. The connection coefficients between two bases diagonalized by two subsequent algebras in the sequence will be Racah polynomials. From this, we can calculate the connection coefficients between any two labeling Abelian algebras. Let us give an example.
\begin{example}\label{bivariate}
Let $\mathcal{Y}_1=\langle C_{12}, C_{123}\rangle$ and $\mathcal{Y}_2=\langle C_{34}, C_{234}\rangle$ be two labeling Abelian algebras of $\mathcal{R}_4$. Consider the following sequence of labeling Abelian algebras:
\[
\langle C_{12}, C_{123}\rangle , \langle C_{23}, C_{123}\rangle, \langle C_{23}, C_{234}\rangle, \langle C_{34}, C_{234} \rangle.
\]
Let $\{ \psi_{k_1, k_2}\}$, $\{ \phi^1_{\ell_1, \ell_2}\}$,   $\{ \phi^2_{\ell_1, \ell_2}\}$ and $\{ \varphi_{s_1, s_2}\}$ be the bases diagonalized by the algebras in the sequence respectively.
We calculate the connection coefficients:
\begin{align*}
\varphi_{s_1, s_2}	&=\sum_{\ell_1} R_{s_1 \ell_1}(C_2,C_3,C_4,C_{234})\phi^2_{\ell_1, s_2} \\
			&=\sum_{\ell_1,k_2} R_{s_1 \ell_1}(C_2,C_3,C_4,C_{234}) R_{s_2 k_2}(C_1, C_{23}, C_{4}, C_{1234}) \phi^1_{\ell_1, k_2} \\
			&=\sum_{k_1,\ell_1 ,k_2} R_{s_1 \ell_1}(C_2,C_3,C_4,C_{234})R_{s_2 k_2}(C_1, C_{23}, C_{4}, C_{1234}) \\
			&\qquad \qquad \times R_{\ell_1k_1}(C_1,C_2,C_3,C_{123}) \psi_{k_1, k_2}.
\end{align*}

The first equality is obtained by considering the eigenspaces of $C_{234}$ which acts as a representation space for $\mathcal{R}^{2,3,4}(3)$. The second equality is obtained by considering the eigenspaces of $C_{23}$ which acts as a representation space for $\mathcal{R}^{1,23,4}(3)$ and the final equality is obtained by considering the eigenspaces of $C_{123}$ which acts as a representation space for $\mathcal{R}^{1,2,3}(3)$. The connection coefficients are
\[
R_{\vec s \vec k}=\sum_{\ell_1} R_{s_1 \ell_1}(C_2,C_3,C_4,C_{234})R_{s_2 k_2}(C_1, C_{23}, C_{4}, C_{1234})R_{\ell_1k_1}(C_1,C_2,C_3,C_{123}).
\]
\end{example}
One can ask the question if for any pair of labeling Abelian algebras we are able to find connection coefficients. To answer this question we introduce the connection graph. The vertices represent labeling Abelian algebras and there is an edge between two vertices if the labeling Abelian algebras differ by one generator. For $\mathcal{R}_4$ the connection graph look as follows:
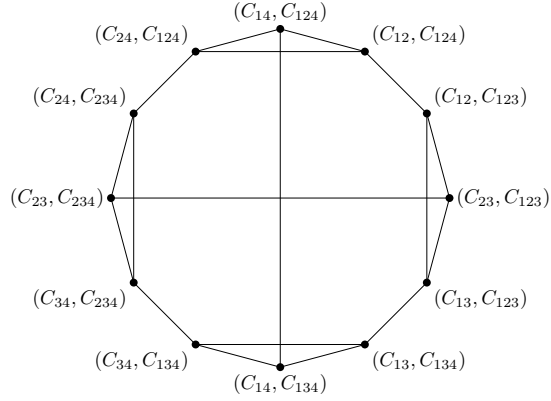
\begin{center}
	\scalebox{.75}{
\begin{tikzpicture}
\draw (30:3cm)--(60:3cm) ;
\draw (60:3cm)--(90:3cm);
\draw (90:3cm)--(120:3cm);
\draw (120:3cm)--(150:3cm);
\draw (150:3cm)--(180:3cm) ;
\draw (180:3cm)--(210:3cm);
\draw (210:3cm)--(240:3cm);
\draw (240:3cm)--(270:3cm);
\draw (270:3cm)--(300:3cm) ;
\draw (300:3cm)--(330:3cm);
\draw (330:3cm)--(0:3cm);
\draw (0:3cm)--(30:3cm);
\fill (30:3cm) circle [radius=2pt] node[anchor=south west] {$(C_{12},C_{123})$};
\fill (60:3cm) circle [radius=2pt] node[anchor=south west] {$(C_{12},C_{124})$};
\fill (90:3cm) circle [radius=2pt] node[anchor=south] {$(C_{14},C_{124})$};
\fill (120:3cm) circle [radius=2pt] node[anchor=south east] {$(C_{24},C_{124})$};
\fill (150:3cm) circle [radius=2pt] node[anchor=south east] {$(C_{24},C_{234})$};
\fill (180:3cm) circle [radius=2pt] node[anchor=east] {$(C_{23},C_{234})$};
\fill (210:3cm) circle [radius=2pt] node[anchor=north east] {$(C_{34},C_{234})$};
\fill (240:3cm) circle [radius=2pt] node[anchor=north east] {$(C_{34},C_{134})$};
\fill (270:3cm) circle [radius=2pt] node[anchor=north] {$(C_{14},C_{134})$};
\fill (300:3cm) circle [radius=2pt] node[anchor=north west] {$(C_{13},C_{134})$};
\fill (330:3cm) circle [radius=2pt] node[anchor=north west] {$(C_{13},C_{123})$};
\fill (0:3cm) circle [radius=2pt] node[anchor=west] {$(C_{23},C_{123})$};
\draw (180:3cm)--(0:3cm);
\draw (90:3cm)--(270:3cm);
\draw (60:3cm)--(120:3cm);
\draw (150:3cm)--(210:3cm);
\draw  (240:3cm) --(300:3cm) ;
\draw  (330:3cm)--(30:3cm) ;
\end{tikzpicture}}
\end{center}
This graph is connected so we can find connection coefficients between any two labeling Abelian algebras of $\mathcal{R}_4$ by the method shown in Example \ref{bivariate}. One can show that the connection graph is connected for the Racah algebra of any rank. For a proof see \cite{DeBie&Genest$vandeVijver&Vinet}.
Other methods to find the connection coefficients have been presented before. See \cite{Scar} for the tree method as well as \cite{Vanderjeugt,Vanderjeugt-2003}.

\begin{example}\label{MultRacah}
In this example we show how to obtain the multivariate Racah polynomials as defined by Tratnik \cite{Tratnik-1991}. Let $\mathcal{Y}_{\text{initial}}=\langle C_{12}, C_{123}, \dots , C_{[n-1]} \rangle$ and $\mathcal{Y}_{\text{final}}=\langle C_{23}, C_{234}, \dots , C_{[2..n]} \rangle$. A sequence of intermediate algebras, each differing by one element with the next, is given as follows:
\[
\mathcal{Y}_i=\langle C_{23}, \dots, C_{[2..2+i]}, C_{[2+i]}, \dots, C_{[n-2]} \rangle.
\]
The connection coefficients up to a gauge constant will be given by
\[
R_{\vec k \vec s}=\prod_{i=1}^{n-2} R_{k_i s_i}(C_1, C_{[2..i+1]}, C_{i+2}, C_{[i+2]}).
\]
If the action of the central elements on the irreducible representation is given by $C_1=\kappa(0,\beta_0)$, $C_{i+1}=\kappa(0,\beta_{i}-\beta_{i-1}-1)$ for some number $\beta_0, \dots, \beta_{n-1}$, then these connection coefficients are written explicitly as
\begin{align*}
 	R_{\vec k \vec s }
 	&=\prod_{j=1}^{n-2} r_{k_j}(2|\vec k|_{j-1}+\beta_j-\beta_0-1,\beta_{j+1}-\beta_j-1,|\vec k|_{j-1}-s_{j+1}-1, \\
 	& \qquad |\vec k|_{j-1}+\beta_j+s_{j+1},-|\vec k|_{j-1}+s_j) ,
 \end{align*}
where $|\vec k|_{j}=\sum_{i=1}^{j}k_i$. This result was obtained in \cite{DeBie&vandeVijver}. These are exactly the multivariate Racah polynomials as defined by Tratnik \cite{DeBie&vandeVijver,Geronimo&Iliev-2010, Tratnik-1991}. In the rank two case one finds bivariate Racah polynomials which coincides with the result in \cite{Post}.
\end{example}

\begin{remark}\label{ConnOP}
As we shall see in the next section, the Racah algebra has realizations in terms of differential operators. After an appropriate gauge transformation, these operators preserve the space of polynomials \cite{KMT}. The common eigenfunctions of the labeling Abelian algebras become multivariable Jacobi polynomials which are mutually orthogonal with respect to the Dirichlet distribution. The coefficients $R_{\vec k \vec s}$ connect different bases of Jacobi polynomials, obtained by an appropriate action of the symmetric group, which preserves the Dirichlet distribution. Within this framework, Lemma \ref{overlap} says that the entries of the transition matrix between different bases of two-variable Jacobi polynomials for the Dirichlet distribution of order $3$ can be expressed in terms of the Racah polynomials, and this was proved by Dunkl \cite{DuOP}. The extensions to arbitrary dimension and, in particular, techniques to compute $R_{\vec k \vec s}$, different relations and the explicit formula for the cyclic permutation in Example \ref{MultRacah} were obtained in \cite{Iliev&Xu-2017}.
\end{remark}

\section{Realizations of the higher rank Racah algebra}\label{S4}

\subsection{The generic superintegrable system on the sphere}
The generic superintegrable system is already well studied. See for example \cite{Genest&Vinet&Zhedanov-2014-3,Genest&Vinet&Zhedanov-2014-2,Iliev-2017,Iliev-2018,Kalnins&Miller&Post-2007,Kalnins&Miller&Post-2011,Kalnins&Miller&Post-2013}. We will introduce this model here.

Let  $\mathbb{S}^{n-1}=\{ (y_1,\dots, y_n) \in \mathbb{R}^n\, |\, \sum_i y_i^2=1 \}$ be the sphere in $\mathbb{R}^n$. We have $n$ variables $y_i$ and we denote $\partial_i:=\partial_{y_i}$. The generic superintegrable system on the sphere is the quantum system with Hamiltonian
\[
 \mathcal{H}=\Delta_{LB}+\sum_{i=1}^n \frac{b_i}{y_i^2}.
\]
The parameters $b_i$ are real numbers. The operator $\Delta_{LB}$ is the Laplace-Beltrami operator:
\[
\Delta_{LB}=\sum_{1\leq i<j \leq n} (y_i\partial_j-y_j\partial_i)^2.
\]
We will show that the Racah algebra can be realized as the symmetry algebra of the Hamiltonian $\mathcal{H}$.
To this end we consider for each $i$ the following realization of $\mathfrak{su}(1,1)$:
\[ 
J_{+,i}=\frac{y_i^2}{2}, \quad J_{-,i}=\frac{1}{2}\left(\partial_i^2+\frac{b_i}{y_i^2}\right), \quad J_{0,i}=\frac{1}{4}(2y_i\partial_i+1). 
\]
The $b_i$ are arbitrary constants. We will construct the higher rank Racah algebra in this realization as explained in section \ref{Rn}.  We need to find the generators. By Lemma \ref{LinearDependence} it suffices to give the generators $C_{ij}$ and $C_i$. An easy calculation shows that $C_i=-\frac{3+4b_i}{16}$. We calculate $C_{ij}$.
\begin{align*}
C_{ij}&=J_{0,ij}^2-J_{0,ij}-J_{+,ij}J_{-,ij} \\
	&=\frac{1}{16}(2y_i\partial_i+2y_j\partial_j+2)^2-\frac{1}{4}(2y_i\partial_i+2y_j\partial_j+2)-\frac{y_i^2+y_j^2}{4}\left(\partial_i^2+\partial_j^2+\frac{b_i}{y_i^2}+\frac{b_j}{y_j^2}\right) \\
	&=-\frac{1}{4}\left((y_i\partial_j-y_j\partial_i)^2+\frac{b_iy_j^2}{y_i^2}+\frac{b_jy_i^2}{y_j^2}+b_i+b_j+1\right).
\end{align*}
Observe that the operators $C_{ij}$ coincide up to a constant with the operators $\mathcal{H}_{i,j}$ defined in \cite{Iliev-2018}. These operators $\mathcal{H}_{i,j}$ are symmetries of the Hamiltonian $\mathcal{H}$. We will show independently that the operators $C_{ij}$ are symmetries of the Hamiltonian. We consider the operator  $C_{[n]}$. By Lemma \ref{Commutative} this operator is central. We calculate an explicit expression for $C_{[n]}$ by using Lemma \ref{LinearDependence} and the fact we are working on the sphere $\mathbb{S}^{n-1}$ so $ \sum_i y_i^2=1$.
\begin{align*}
C_{[n]}	&=\sum_{\{i,j\}\subset [n]} C_{ij}-(n-2)\sum_{i=1}^n C_i \\
		&=-\frac{1}{4}\,\sum_{\{i,j\}\subset [n]} \left( (y_i\partial_j-y_j\partial_i)^2+\frac{b_iy_j^2}{y_i^2}+\frac{b_jy_i^2}{y_j^2}+b_i+b_j+1\right)+(n-2)\sum_{i=1}^n  \frac{3+4b_i}{16} \\
		&=-\frac{1}{4}\left(\Delta_{LB}+\frac{n(n-1)}{2}+\sum_{i \neq j}  \frac{b_iy_j^2}{y_i^2}+(n-1)\sum_{i=1}^n b_i\right)+(n-2)\sum_{i=1}^n \frac{3+4b_i}{16}  \\
		&=-\frac{1}{4}\left(\Delta_{LB}+\frac{n(n-1)}{2}+\sum_{i , j}  \frac{b_iy_j^2}{y_i^2}+(n-2)\sum_{i=1}^n b_i\right)+(n-2)\sum_{i=1}^n \frac{3+4b_i}{16} \\
		&=-\frac{1}{4}\left(\Delta_{LB}+\frac{4n-n^2}{4}+\sum_{i}  \frac{b_i}{y_i^2} \sum_j y_j^2 \right)\\
		&=-\frac{1}{4}\left(\Delta_{LB}+\sum_{i}  \frac{b_i}{y_i^2} \right)+\frac{n^2-4n}{16}\\
		&=-\frac{1}{4}\mathcal{H}+\frac{n^2-4n}{16}.
\end{align*}
It follows that as $C_{ij}$ commutes with $C_{[n]}$ it also commutes with $\mathcal{H}$. The Racah algebra thus governs the symmetries of the Hamiltonian $\mathcal{H}$. We conclude that the higher rank Racah algebra acts as a symmetry algebra for this Hamiltonian $\mathcal{H}$ from which we can derive the integrals of motion.

\begin{remark}\label{Re:AlgGen}
One can show that the $\binom{n}{2}$ operators $C_{ij}$ are linearly independent, but are not algebraically independent when $n>3$. However, the $2n-3$ operators in the set $G=\{C_{1,j}:j=2,\dots,n\}\cup\{C_{i,n}:i=2,\dots,n-1\}$ generate the symmetry algebra for the Hamiltonian $\mathcal{H}$, and every operator $C_{ij}$ can be written explicitly as a polynomial of the operators in $G$, see \cite{Iliev-2017,Iliev-2018}. Moreover, these constructions can be generalized for a discrete extension of the generic superintegrable system on the sphere related to the Hahn polynomials and the hypergeometric distribution \cite{Iliev&Xu-2017-2}. 
\end{remark}

\subsection{The Dunkl model} For a detailed exposition of the Dunkl model see \cite{DeBie&Genest$vandeVijver&Vinet}. 
The Dunkl model is obtained by realizing the algebra $\mathfrak{su}(1,1)$ using the Dunkl operators as defined by C.F. Dunkl in \cite{DTAMS}.
We consider the Dunkl operators related to the reflection group $\mathbb{Z}_2^n$. They are defined as follows:
\[
  T_{i}:=\partial_{x_i}+\mu_i\frac{1-R_{i}}{x_i}.
\]
The operator $R_{i}$ is the reflection which acts as $R_{i}(f(x_i))=f(-x_i)$. The number $\mu_i>0$ is a deformation parameter. The operators $T_i$ are commutative. With these operators one can realize the algebra $\mathfrak{su}(1,1)$:
\[
 J_{+,i}=\frac{x_i^2}{2}, \quad J_{-,i}=\frac{T_i^2}{2}, \quad J_{0,i}=\frac{1}{2}\left(x_i\partial_i+\mu_i+\frac{1}{2}\right).
\]
This realization of $\mathfrak{su}(1,1)$ leads to a new realization of the Racah algebra $\mathcal{R}_n$ with
\begin{align*}
C_i&=\frac{4\mu_i^2-4\mu_i R_i-3}{16}, \\
C_{ij}&=\frac{1}{4}\left( -(x_iT_j-x_jT_i)^2+(\mu_iR_j+\mu_jR_i)^2-1\right).
\end{align*}
Observe that the operator  $x_iT_j-x_jT_i$ is the Dunkl angular momentum operator. By Lemma \ref{sl2} this realization of the Racah algebra is in the centralizer of the algebra generated by the following elements:
\[
J_{+,[n]}=\frac{1}{2}\sum_{i=1}^n x_i^2, \quad J_{-,[n]}=\frac{1}{2}\sum_{i=1}^n T_i^2, \quad J_{0,[n]}=\frac{1}{2}\sum_{i=1}^n \left(x_i\partial_i+\mu_i+\frac{1}{2}\right)
\]
Observe that $J_{-,[n]}$ is the Dunkl-Laplacian $\Delta_{\text{Dunkl}}$ times $1/2$. The Racah algebra thus acts as a symmetry algebra for the Dunkl-Laplacian. Also observe that the Euler operator $\mathbb{E}_n=\sum_{i=1}^n x_i\partial_i$ appears in $J_{0,[n]}$.
 Let $\mathcal{P}_k$ be the set of homogeneous polynomials of degree $k$. These are eigenspaces of  $J_{0,[n]}$. Consider the space of Dunkl-harmonics $\mathcal{H}_k:=\mathcal{P}_k \cap \ker(\Delta_{\text{Dunkl}})$. These spaces will act as representations for the higher rank Racah algebra.

\subsection{The Barut-Girardello model}
For a detailed overview of the Barut-Girardello model we refer to the following article \cite{DeBie&Iliev&Vinet}. The Barut-Girardello model for the rank one Racah algebra was previously considered in \cite{Genest&Vinet&Zhedanov-2014}.
The previous models realize the Racah algebra $\mathcal{R}_n$ in $n$ variables. The Barut-Girardello model has the interesting property that it realizes the Racah algebra in a number of variables equal to the rank $n-2$ of said algebra. This is obtained as follows. Consider the following realization of $\mathfrak{su}(1,1)$:
\[
 J_+=x^2\partial_x+2\nu x ,\quad J_-=\partial_x, \quad J_0=x\partial_x+\nu.
\]
After introducing $n$ variables $x_1. \dots, x_n$, their partial derivatives $\partial_1, \dots \partial_n$  and $n$ parameters $\nu_1,\dots, \nu_n$ one constructs a realization of the Racah algebra $\mathcal{R}_n$. It is the centralizer of the following $\mathfrak{su}(1,1)$ algebra by Lemma \ref{sl2}.
\[
  J_{+,[n]}=\sum_{i=1}^n (x_i^2\partial_{x_i}+2\nu_i x_i) ,\quad J_{-,[n]}=\sum_{i=1}^n\partial_i, \quad J_{0,[n]}=\sum_{i=1}^n (x_i\partial_i+\nu_i).
\]
Let $\mathcal{H}_k(\mathbb{R}^n)$ be the kernel of $J_{-,[n]}$ in the space of homogenous polynomials defined on $\mathbb{R}^n$. It has the following basis:
\[
 \varphi_{j_1, \ldots, j_{n-2}}= (x_1 - x_2)^k u_1^{j_1} u_2^{j_2} \ldots u_{n-2}^{j_{n-2}}
\]
with
\[
u_j := \frac{x_{j+2} - x_{j+1}}{x_1 - x_2}, \qquad j \in \{ 1, \ldots, n-2 \}.
\]
When one gauges the Racah algebra as follows
\begin{equation}
\label{gauge}
\widetilde{C_B} = (x_1 - x_2)^{-k} C_B (x_1 - x_2)^{k},
\end{equation}
one obtains an algebra acting on polynomials of at most degree $k$ in the variables $u_1,\dots u_{n-2}$. Explicit calculation leads to the following realization:
\begin{theorem}
The space $\Pi_k^{n-2}$ of all polynomials of degree $k$ in $n-2$ variables carries a realization of the rank $n-2$ Racah algebra $\mathcal{R}_n$.
This realization is given explicitly by
\[
\widetilde{C_{i}}=\nu_i(\nu_i-1), \qquad i \in [n]
\]
and, for $i, j \in \{3, \ldots, n \}$,
\begin{align*}
\widetilde{C_{12}}&=- \left(k-1-\sum_{\ell=1}^{n-2} u_{\ell}\partial_{u_\ell} \right)  \left(-k-\partial_{u_1}+\sum_{\ell=1}^{n-2} u_{\ell}\partial_{u_\ell} \right) + 2 \nu_2 \left(k-\sum_{\ell=1}^{n-2} u_{\ell}\partial_{u_\ell} \right) \\
& \qquad - 2 \nu_1\left(-k-\partial_{u_1}+\sum_{\ell=1}^{n-2} u_{\ell}\partial_{u_\ell} \right) + (\nu_1+\nu_2)(\nu_1+\nu_2-1)\\
\widetilde{C_{1j}}&=- \left(1 - \sum_{\ell=1}^{j-2} u_{\ell} \right)^2 \left(k-1-\sum_{\ell=1}^{n-2} u_{\ell}\partial_{u_\ell} \right)   \left( \partial_{u_{j-2}}- \partial_{u_{j-1}} \right)\\
& \qquad + 2 \nu_j \left(1 - \sum_{\ell=1}^{j-2} u_{\ell} \right)\left(k-\sum_{\ell=1}^{n-2} u_{\ell}\partial_{u_\ell} \right) 
 - 2 \nu_1  \left(1 - \sum_{\ell=1}^{j-2} u_{\ell} \right)  \left( \partial_{u_{j-2}}- \partial_{u_{j-1}} \right)\\ & \qquad  + (\nu_1+\nu_j)(\nu_1+\nu_j-1)\\
\widetilde{C_{2j}}&= -\left( \sum_{\ell=1}^{j-2} u_{\ell} \right)^2 \left(1-k-\partial_{u_1}+\sum_{\ell=1}^{n-2} u_{\ell}\partial_{u_\ell} \right)   \left( \partial_{u_{j-2}}- \partial_{u_{j-1}} \right)\\
& \qquad + 2 \nu_j \left( \sum_{\ell=1}^{j-2} u_{\ell} \right) \left(k+\partial_{u_1}-\sum_{\ell=1}^{n-2} u_{\ell}\partial_{u_\ell} \right)  + 2 \nu_2 \left( \sum_{\ell=1}^{j-2} u_{\ell} \right)\left( \partial_{u_{j-2}}- \partial_{u_{j-1}} \right)   \\
& \qquad + (\nu_2+\nu_j)(\nu_2+\nu_j-1)\\
\widetilde{C_{ij}}&= - \left( \sum_{\ell=j-1}^{i-2} u_{\ell} \right)^2 \left( \partial_{u_{i-2}}- \partial_{u_{i-1}} \right) \left( \partial_{u_{j-2}}- \partial_{u_{j-1}} \right)\\
 & \qquad+ 2 \nu_j  \left( \sum_{\ell=j-1}^{i-2} u_{\ell} \right)  \left( \partial_{u_{i-2}}- \partial_{u_{i-1}} \right) - 2 \nu_i  \left( \sum_{\ell=j-1}^{i-2} u_{\ell} \right)  \left( \partial_{u_{j-2}}- \partial_{u_{j-1}} \right)\\
& \qquad + (\nu_i+\nu_j)(\nu_i+\nu_j-1)
\end{align*}
where we assume $i >j$ and with $u_{n-1}=0$ whenever it appears.
\end{theorem}
This model has only $n-2$ variables which is equal to the rank of the higher rank Racah algebra. It is possible to embed this algebra into a differential operator realization of the universal enveloping algebra of $\mathfrak{sl}_{n-1}$, see \cite{DeBie&vandeVijver$Vinet}.

\subsection{The discrete model} For a detailed overview see \cite{DeBie&vandeVijver,Iliev-2017}.
The discrete model has no known underlying realization of $\mathfrak{su}(1,1)$. Instead its action is derived from the action on any irreducible representation of $\mathcal{R}_n$ denoted by $V$.
Let $\mathcal{Y}_{\text{initial}}$ and $\mathcal{Y}_{\text{final}}$ be two labeling Abelian algebras (Definition \ref{label}). Consider two bases of $V$: $\{ \psi_{\vec k}\}$ diagonalizing  $\mathcal{Y}_{\text{initial}}$ and $\{ \varphi_{\vec k} \}$ diagonalizing $\mathcal{Y}_{\text{final}}$ as in Example \ref{MultRacah}. The connection coefficients between these two bases are given by the functions $R_{\vec k}$ defined as:
\[
\bra{\varphi_{\vec s}}\ket{\psi_{ \vec k}}=:R_{\vec s}(\vec k) .
\]
By Example \ref{MultRacah} these functions $R_{\vec k}$  are multivariate Racah polynomials. We define the action of a generator $C_A$ of $\mathcal{R}_n$ by
\[
C_AR_{\vec s}(\vec k)=C_A \bra{\varphi_{\vec s}}\ket{\psi_{ \vec k}}:=\bra{\varphi_{\vec s}}C_A\ket{\psi_{ \vec k}}.
\]
We want to describe the operators $C_A$ acting of the function $R_{\vec k}$. To do so we identify operators whose action on these function $R_{\vec k}$ coincides with the action of $C_A$. The multivariate Racah polynomials are eigenvectors of the labeling Abelian algebras $\mathcal{Y}_{\text{initial}}$ and $\mathscr{Y}_{\text{final}}$  in this realization. The multivariate Racah polynomials are also eigenvectors of the following Racah operators. These were originally introduced in \cite{ Geronimo&Iliev-2010}. See also \cite{Iliev-2017,Iliev-2018,Iliev&Xu-2017-2,Post}.
\begin{definition} Put
\[ 
\mathcal{L}_{j}=\sum_{\substack{\vec \nu \in \{ -1,0,1\}^{j}\\ \vec \nu\neq \vec 0}} G_{\vec \nu}( E_{\vec \nu}-1).
\]
Here $E_{\vec \nu}$ is a shift operator defined as follows. Let $E_{x_i}^{\nu_i}(f(x_j))=f(x_j+\delta_{ij}\nu_i)$. Then we define $E_{\vec \nu}=E_{x_1}^{\nu_1}E_{x_2}^{\nu_2}\ldots E_{x_{j}}^{\nu_{j}}$.
The $G_{\vec \nu}$ are rational functions in the variables $x_0, x_1, \ldots , x_{j+1}$ and $\beta_0, \ldots, \beta_{j+1}$ and are defined as follows. We introduce the following functions:
\begin{align*}
B_i^{0,0}&:=x_i(x_i+\beta_i)+x_{i+1}(x_{i+1}+\beta_{i+1})+\frac{(\beta_i+1)(\beta_{i+1}-1)}{2}, \\
B_i^{0,1}&:=(x_{i+1}+x_i+\beta_{i+1})(x_{i+1}-x_i+\beta_{i+1}-\beta_i), \\
B_i^{1,0}&:=(x_{i+1}-x_i)(x_{i+1}+x_{i}+\beta_{i+1}),\\
B_i^{1,1}&:=(x_{i+1}+x_i+\beta_{i+1})(x_{i+1}+x_i+\beta_{i+1}+1).
\end{align*}
Let $I_i f(x_i):=f(-x_i-\beta_i)$. We extend $B^{s,t}$ by defining:
\begin{align*}
B_i^{-1,t}&:=I_i(B_i^{1,t}), \\
B_i^{s,-1}&:=I_{i+1}(B_i^{s,1}), \\
B_i^{-1,-1}&:=I_i(I_{i+1}(B_i^{1,1})).
\end{align*}
We also introduce
\begin{align*}
b_i^0&:=(2x_i+\beta_i+1)(2x_i+\beta_i-1), \\
b_i^1&:=(2x_i+\beta_i+1)(2x_i+\beta_i), \\
b_i^{-1}&:=I_{i}(b_i^{1}).
\end{align*}
Let $|\vec \nu|_0$ be the number of zeroes appearing in $\vec \nu$. Then $G_{\vec \nu}$ is
\[
G_{\vec \nu}:=2^{|\vec \nu|_0}\frac{\prod_{i=0}^{j} B_i^{\nu_i,\nu_{i+1}}}{\prod_{i=1}^{j} b_i^{\nu_i}}.
\]
\end{definition}
The action of these Racah operators coincides up to scalar with the action of $\mathscr{Y}_{\text{final}}$. 
In general it can be shown that the action of any $C_{[i..j]}$ coincides with the action of minus a Racah operator up to the addition of a scalar. These operators generate the Racah algebra $\mathcal{R}_n$ as
\[
C_{ij}=C_{[i..j]}-C_{[i..j-1]}-C_{[i+1..j]}+C_{[i+1..j-1]}+C_i+C_j.
\]
To present the discrete model of the higher rank Racah algebra we need to introduce the following map:
\begin{definition}
 Let $\sigma$ be the map that adds $1$ to any index of an expression:
\begin{align*}
  \textup{Alg}[x_0,  \ldots, x_s;\beta_0, \ldots \beta_s; E_{1}, \ldots , E_{s}] &\rightarrow \textup{Alg}[x_1, \ldots, x_{s+1};\beta_1, \ldots \beta_{s+1}; E_{2}, \ldots , E_{s+1}] \\
 \sigma(x_i)&=x_{i+1}, \\
 \sigma(\beta_i)&=\beta_{i+1},\\
 \sigma(E_{x_i})&=E_{x_{i+1}} ,
 \end{align*}
e.g. $\sigma(x_1\beta^2_2E_{x_1})=x_2\beta^2_3E_{x_2}$.
\end{definition}

In \cite{DeBie&vandeVijver} the following theorem was proven.
\begin{theorem}\label{MainTheorem}
Define $\kappa(x,\beta)=\left(x+\frac{\beta+1}{2}\right)\left(x+\frac{\beta-1}{2}\right)$. With $\mathcal{L}_i$ given as above, define the following operators:
\begin{align}
C_{[m]}&=\kappa(x_{m-1},\beta_{m-1}), \label{OperatorOne} \\
C_{[2\ldots m+1]}&=-\mathcal{L}_{m-1}+\kappa(0,\beta_m-\beta_0-1), \label{OperatorTwo}\\
C_{[p\ldots q]}&=\sigma^{p-2}(C_{[2\ldots q-p+2]}) \label{OperatorThree}, \qquad \text{ if } \quad p>2
\end{align}
and set $x_0=0$.
The algebra generated by these operators is a discrete realization of $R(n)$.
\end{theorem}
This result coincides with the realization in rank one which was already known \cite{Genest&Vinet&Zhedanov-2014-2,Genest&Vinet&Zhedanov-2014-3} and the result obtained in \cite{Post} for the rank $2$ case. In \cite{Iliev-2017}, the representation of $R(n)$ in Theorem~\ref{MainTheorem} was constructed by defining in terms of the Racah operators the generators $\{C_{1,j}:j=2,\dots,n\}\cup\{C_{i,n}:i=2,\dots,n-1\}$ of the symmetry algebra discussed in Remark \ref{Re:AlgGen}.

\section{Further results and conclusions}\label{Conclusion}
The algebraic properties of the higher rank Racah algebra are not well understood yet. The representation theory for the rank one case is being build up in \cite{Hau-Wen&Bockting-Conrad-3,Hau-Wen&Bockting-Conrad,Hau-Wen&Bockting-Conrad-2,Bu&Hou&Gao,Huang}.  The relationship with other algebras is also being studied. The rank one Racah algebras have the Temperly-Lieb and Brauer algebras as quotients, see \cite{Crampe$Poulain&Vinet}. It would be interesting to see this result generalized to higher rank. Howe type dualities have been brought to light involving the higher rank Racah algebra, see \cite{Gaboriaud&Vinet&Vinet&Zhedanov}.
When one replaces $\mathfrak{su}(1,1)$ with different algebras in the method given in Section \ref{Rn}, one obtains algebras which are strongly related to the higher rank Racah algebra and for which we can solve the Racah problem. If one takes the Lie super algebra $\mathfrak{osp}(1|2)$, one obtains the higher rank Bannai-Ito algebra \cite{DeBie&Genest&Vinet-2016-2,DeBie&Genest$vandeVijver&Vinet-2}. The $q$-deformation $\mathfrak{osp}_q(1|2)$ leads to the higher rank $q$-deformed Bannai-Ito algebra \cite{DeBie&DeClercq&vandeVijver}. Lastly if one replaces $\mathfrak{su}(1,1)$ with the oscillator algebra or equivalently the Heisenberg algebra one obtains an algebra which contains the Lie algebra $\mathfrak{sl}_n$ \cite{Crampe&vandeVijver&Vinet}. This suggest that there must be a deep connection between the higher rank Racah algebra and the special linear Lie algebra $\mathfrak{sl}_n$.

\section{Acknowledgements}
PI is partially supported by the Simons Foundation Grant \#635462.
The work of HDB and WVDV is supported by the Research Foundation Flanders (FWO) under Grant EOS 308894451. 
The research of LV is supported in part by a Discovery Grant from the Natural Sciences and Engineering Research Council (NSERC) of Canada.
\bibliographystyle{amsplain}

\end{document}